\documentclass[12pt,reqno]{amsart}

\usepackage[a4paper]{geometry}
\usepackage{graphicx}%
\usepackage{multirow}%
\usepackage{amsmath,amssymb,amsfonts}%
\usepackage{amsthm}%
\usepackage{mathrsfs}%
\usepackage[title]{appendix}%
\usepackage{xcolor}%
\usepackage{textcomp}%
\usepackage{manyfoot}%
\usepackage{booktabs}%
\usepackage{listings}%
\usepackage{xspace}
\usepackage{lmodern}

\newtheorem{theorem}{Theorem}[section]
\newtheorem{prop}[theorem]{Proposition}
\newtheorem{lemme}[theorem]{Lemma}
\newtheorem{coro}[theorem]{Corollary}

\renewcommand{\L}{\ensuremath{\mathbb{L}}\xspace}
\newcommand{\R}{\ensuremath{\mathbb{R}}\xspace}
\newcommand{\BP}{\ensuremath{\mathbf{P}}\xspace}
\newcommand{\CC}{\ensuremath{\mathcal{C}}\xspace}
\newcommand{\MM}{\ensuremath{\mathcal{M}}\xspace}
\newcommand{\RR}{\ensuremath{\mathcal{R}}\xspace}

\DeclareMathOperator{\Tr}{Tr}
\DeclareMathOperator{\Hess}{Hess}

\DeclareMathOperator{\Ent}{Ent}
\DeclareMathOperator{\id}{id}

\begin{document}
\title[]{Logarithmic Sobolev inequalities for generalised Cauchy measures}

\author[B. Huguet]{Baptiste Huguet} 
\address{IRMAR (ENS Rennes), UMR CNRS 6625, Univ. Rennes, France}
\email{baptiste.huguet@math.cnrs.fr}
\urladdr{https://orcid.org/0000-0003-3211-3387}
\keywords{Curvature-dimension criterion; generalised Cauchy measures; logarithmic Sobolev inequality.}

\begin{abstract}
We prove a curvature-dimension criterion and obtain logarithmic Sobolev inequalities for generalised Cauchy measures with optimal weights and explicit constants. In the one-dimensional case, this constant is even optimal. From these inequalities, we derive concentration results, which allow concluding the case of the pathological dimension two.
\end{abstract}
\maketitle

\section{Introduction}
This article is dedicated to the logarithmic Sobolev inequality for generalised Cauchy measures, by means of curvature-dimension arguments. These measures are defined on $\R^n$ by 
\[\mu_{\beta}(dx) =\frac{1}{Z^{(n)}_{\beta}} (1+|x|^2)^{-\beta} dx,\quad\text{with}\quad Z^{(n)}_{\beta}=\frac{\pi^{n/2}\Gamma(\beta-n/2)}{\Gamma(\beta)},\] 
where $\beta>n/2$.  In many ways, they can be seen as an approximation of Gaussian measures. Indeed, thanks to a rescaling argument and letting the parameter $\beta$ go to $\infty$, we recover the Gaussian distribution. However, this metaphor is more involved as it can be extended to the comparison between the heat kernel and the Barrenblatt kernel or between the heat equation and the fast-diffusion equation.  

Generalised Cauchy measures do not satisfy the classical functional inequalities, such as Poincaré or logarithmic Sobolev. This makes them appropriate examples for studying weighted functional inequalities. The weighted Poincaré inequality has been studied in numerous articles, starting from the founding article about functional inequalities for generalised Cauchy measures, \cite{BobLed}, with different approaches leading to a collection of partial results (see  \cite{ABJ}, \cite{BJM},  \cite{GZ}, \cite{Hug}, \cite{Ngu}, \cite{Saum}). A unified and complete proof of the optimal inequality is brought in \cite{Hug2}, by means of integrated curvature-dimension arguments. 

On the other hand, the logarithmic-Sobolev inequality for generalised Cauchy measures has been less studied. We say that a probability measure, $\nu$, satisfies the logarithmic Sobolev inequality, with weight $\Omega$ and constant $C>0$ if  
\[\Ent_{\nu}(f^2)\leq C\int_{\R^n}|df|^2\Omega\, d\nu,\quad f\in\CC^\infty_c(\R^n).\]
where $\CC^\infty_c(\R^n)$ denotes the set of smooth compactly supported functions. Whenever this inequality holds, it extends to a domain in $\L^2(\nu)$.

In \cite{BobLed},  S. Bobkov and M. Ledoux established the following weighted inequality 
\[\Ent_{\mu_\beta}(f^2)\leq\frac{1}{2\beta-1}\int_{\R^n}|df|^2(1+|x|^2)^2\, d\mu_\beta, \quad \beta> (n+1)/2.\]
Their proof is based on the expression of the generalised Cauchy measure as the distribution of the quotient between independent Gaussian  and  chi-squared random variables. The constant is explicit and has a good order: up to rescaling, they recover Gross' logarithmic Sobolev inequality for Gaussian measures from \cite{Gross}. Their result has been reproved in one dimension by means of intertwining arguments in \cite{BJ}. Nevertheless, their weight is not optimal. Actually, we know, from a Muckenhoupt-type argument, that the weight must be at least of magnitude $(1+|x|^2)\log(|x|)$ (see \cite{BobGotz} or \cite{BartRob} for a refined version of the result). It was proved by \cite{CGW} that these measures do satisfy a logarithmic Sobolev inequality with the weight $(1+|x|^2)\log(e+|x|^2)$, in any dimension. However, their method, based on Lyapunov conditions, cannot provide explicit constants. It is noticeable that the weight is different  from the optimal weight for Poincaré and Beckner inequalities. The aim of our article is to prove an inequality with the optimal weight and explicit, or even optimal, constants.

Let us describe the structure of this article. In Section \ref{Sec:Cauchy}, we introduce the two arguments of this article: the rescaled measures and the curvature-dimension criterion. This criterion will be declined in the following two sections. In section \ref{Sec:1d}, we study the one-dimensional case. We show that the curvature-dimension argument is optimal. In Section \ref{Sec:nd}, we study the general case. We obtain explicit logarithmic-Sobolev inequalities for small parameters. Section \ref{Sec:2d}, is devoted to the special case of dimension two, for which the curvature-dimension argument fails. An explicit inequality is obtained by tensorisation and concentration properties.

\section{Rescaled generalised Cauchy measures}
\label{Sec:Cauchy}
The goal of this section is to introduce  and motivate the two main ingredients of this work. First, we introduce rescaled measures and we explain why it is natural to work with them. Then, we obtain a lower bound on the optimal constant for the logarithmic Sobolev inequality. Finally, we recall the curvature-dimension criterion and we explain how it applies to the rescaled generalised Cauchy measures.

For a fixed $\beta>1/2$ and $\sigma\geq 1$, we defined the rescaled generalised Cauchy measure $\mu_{\beta,\sigma}$  on $\R^n$ by 
\[\mu_{\beta,\sigma}(dx) =\frac{1}{Z^{(n)}_{\beta,\sigma}} (\sigma+|x|^2)^{-\beta} dx,\quad\text{with}\quad Z^{(n)}_{\beta,\sigma}=\frac{\pi^{n/2}\Gamma(\beta-n/2)}{\Gamma(\beta)}\sigma^{n/2-\beta}.\] 
The scaling parameter $\sigma$ can be interpreted as a computation artefact upon which we optimise. In the limit case $\sigma=1$, $\mu_\beta =\mu_{\beta,1}$ is the classical generalised Cauchy measure and will be denoted $\mu_\beta$. On the other hand, the measures $\mu_{\beta, 2\beta}$ converge to the Gaussian distribution when $\beta$ goes to $+\infty$.
We denote by $\omega_{\sigma}$ the function 
\[\omega_{\sigma} :x\in\R^n \mapsto \sigma +|x|^2.\] 
For simplicity, $\omega_1$ is denoted by $\omega$. In this work, we prove an inequality for rescaled measures, $\mu_{\beta,\sigma}$, with the weight $\omega_{\sigma}\log(\omega_{\sigma})$ and a constant $C_\beta$, independent of the scaling parameter $\sigma$. Let us notice that as long as $\sigma>1$, this weight is positive but not any longer for $\sigma=1$. Actually, we can even prove that generalised Cauchy measures do not satisfy any logarithmic Sobolev inequality with weight $\omega\log(\omega)$, at least in dimension $1$, by considering the test functions 
\[f_k(x) = \left\{\begin{array}{lcc}
\sqrt{\frac{3k}{\log(k)}}(1-2kx)_+&\text{if}& x\geq0\\
\sqrt{\frac{3k}{\log(k)}}(1+2kx)_+&\text{if}& x\leq0\\
\end{array}\right.\]
which satisfy
\[\Ent_{\mu_\beta}(f_k^2)\sim \frac{1}{Z_{\beta}},\quad\text{and}\quad \int_\R|f_k'|^2\omega\log(\omega)\,d\mu_\beta\sim\frac{1}{Z_{\beta}\log(k)}.\]
However, an appropriate change of variables allows results on rescaled measures to be transferred to generalised Cauchy measures. 

\begin{lemme}\label{prop:rescale}
Assume that $\mu_{\beta,\sigma}$ satisfies a logarithmic Sobolev inequality with weight $\omega_{\sigma}\log(\omega_{\sigma})$ and constant $C>0$, then $\mu_\beta$ satisfies a logarithmic Sobolev inequality with weight $\omega(\log(\sigma) + \log(\omega))$ and the same constant $C$.
\end{lemme}

\begin{proof}
Let $f\in\CC^\infty_c(\R^n)$, we have
\begin{align*}
\int_{\R^n}f\, d\mu_{\beta,\sigma}
&= \int_{\R^n} f(x)(\sigma+|x|^2)^{-\beta}\frac{1}{Z_{\beta,\sigma}}\, dx\\
&= \int_{\R^n} f(\sqrt{\sigma}y)(1+|y|^2)^{-\beta}\frac{\sigma^{n/2-\beta}}{Z_{\beta,\sigma}}\, dy\\
&= \int_{\R^n}\tilde{f}\, d\mu_{\beta}
\end{align*}
where $\tilde{f} = f(\sqrt{\sigma}\cdot)$.
In particular, $\Ent_{\mu_{\beta,\sigma}}(f^2) = \Ent_{\mu_\beta}(\tilde{f}^2)$. By the same argument, we have
\begin{align*}
\int_{\R^n} |d f|^2\omega_\sigma\log(\omega_\sigma)\, d\mu_{\beta, \sigma}
&=\int_{\R_n} |d f|^2(\sqrt{\sigma} y) (\sigma +\sigma|y|^2)\log(\sigma + \sigma|y|^2)\, \mu_{\beta}(dy)\\
&=\int_{\R_n} |d \tilde{f}|^2( y) (1 +|y|^2)\left(\log(\sigma) + \log(1+|y|^2)\right)\, \mu_{\beta}(dy)\\
\end{align*}
Hence, the result is proved.
\end{proof}

This lemma shows that it is sufficient to work on the rescaled measures.  Since we want to control the sharpness of our inequalities, we need a lower bound on the optimal constant. To do this, it is sufficient to have a suitable test function, or a suitable sequence of test functions, that blow up in $\L^2(\mu_{\beta,\sigma})$. 

\begin{prop}\label{prop:lower}
For all $n\geq1$, $\sigma>1$ and $\beta>n/2$, if $\mu_{\beta,\sigma}$ satisfies a logarithmic Sobolev inequality with weight $\omega_\sigma\log(\omega_\sigma)$ and constant $C$, then we have \[C\geq 2/(2\beta-n).\]
\end{prop}

\begin{proof}
For $\varepsilon>0$, we define $f_\varepsilon=\omega_\sigma^\varepsilon$. These functions are square-integrable if and only if $\varepsilon<(2\beta-n)/4=\varepsilon_0$ and in this case, we have

\begin{align*}
\Ent_{\mu_{\beta,\sigma}}(f_\varepsilon^2) =& \frac{Z_{\beta-2\varepsilon, \sigma,}}{Z_{\beta,\sigma}}\left(2\varepsilon\int_{\R^n}\log(\omega_\sigma)\, d\mu_{\beta-2\varepsilon, \sigma} - \log\left(\frac{Z_{\beta-2\varepsilon, \sigma}}{Z_{\beta, \sigma}}\right)\right)\\
\int_{\R^n}|df_\varepsilon|^2\omega_\sigma\log(\omega_\sigma)\, d\mu_{\beta,\sigma} =& 4\varepsilon^2\frac{Z_{\beta-2\varepsilon, \sigma}}{Z_{\beta, \sigma}} \left( \int_{\R^n}\log(\omega_\sigma)\, d\mu_{\beta-2\varepsilon, \sigma}\right.\\
&\left. - \sigma\frac{Z_{\beta+1-2\varepsilon, \sigma}}{Z_{\beta-2\varepsilon, \sigma}} \int_{\R^n}\log(\omega_\sigma)\, d\mu_{\beta+1-2\varepsilon, \sigma}\right)
\end{align*} 

These two quantities diverge when $\varepsilon\to\varepsilon_0$, but we are only interested in their quotient.  If $\mu_{\beta,\sigma}$ satisfies the inequality with weight $\omega_\sigma\log(\omega_\sigma)$ and constant $C$, we have

\begin{equation}\label{eq:const}
2\varepsilon-\frac{\log\left(\frac{Z_{\beta-2\varepsilon, \sigma}}{Z_{\beta, \sigma}}\right)}{\int_{\R^n}\log(\omega_\sigma)\, d\mu_{\beta-2\varepsilon, \sigma}} \leq 4\varepsilon^2 C\left(1 -\sigma\frac{Z_{\beta+1-2\varepsilon,\sigma,}}{Z_{\beta-2\varepsilon, \sigma}}\frac{ \int_{\R^n}\log(\omega_\sigma)\, d\mu_{\beta+1-2\varepsilon, \sigma}}{\int_{\R^n}\log(\omega_\sigma)\, d\mu_{\beta-2\varepsilon, \sigma,}}\right).
\end{equation}

When $\varepsilon$ goes to $\varepsilon_0$, the right-hand side converges to $4\varepsilon_0^2 C$. Indeed, $Z_{n/2+1,\sigma}$ is finite and $\log(\omega_\sigma)$ is integrable with respect to $\mu_{n/2+1,\sigma}$, while $Z_{\beta-2\varepsilon, \sigma}$ and $\int_{\R^n}\log(\omega_\sigma)\, d\mu_{\beta-2\varepsilon, \sigma}$ diverge together to $+\infty$. At the left-hand side, we have the quotient of two quantities, which diverge to $+\infty$. So, we need to compare their speed of divergence. On the one hand, we have
\[\log\left(\frac{Z_{\beta-2\varepsilon, \sigma}}{Z_{\beta, \sigma}}\right)\sim\log(\Gamma(\beta-2\varepsilon-n/2))\sim -\log(\alpha),\]
where $\alpha=\beta-2\varepsilon-n/2$ goes to $0$. On the other hand, we have
\begin{align*}
\int_{\R^n}\log(\omega_\sigma)\, d\mu_{\beta-2\varepsilon, \sigma}
=& \frac{1}{Z_{\beta-2\varepsilon,\sigma}} \int_{\R^n}\log(\sigma+|x|^2)(\sigma+|x|^2)^{-\beta+2\varepsilon}\, dx\\
=& \frac{\pi^{n/2}}{\Gamma(n/2+1)Z_{\beta-2\varepsilon,\sigma}} \int_0^{+\infty}\log(\sigma+r^2)(\sigma+r^2)^{-\beta+2\varepsilon}r^{n-1}\, dr\\
\end{align*}
When $\varepsilon$ goes to $\varepsilon_0$, this quantity diverges with the same speed as
\[\frac{1}{\Gamma(\alpha)}\int_1^{+\infty} \log(t)t^{-1-2\alpha}\, dt \asymp \frac{1}{\alpha}.\]
Thus, the quotient \[\frac{\log\left(\frac{Z_{\beta-2\varepsilon, \sigma}}{Z_{\beta, \sigma}}\right)}{\int_{\R^n}\log(\omega_\sigma)\, d\mu_{\beta-2\varepsilon, \sigma}}\] behaves as $\alpha\log(\alpha)$ when $\varepsilon$ goes to $\varepsilon_0$ and so, the left-hand side of equation \eqref{eq:const} converges to $2\varepsilon_0$. At the limit, we get $1\leq 2\varepsilon_0 C$. This ends the proof.
\end{proof}

Note that this lower bound is independent of $\sigma$. Actually, the proof could have been done directly for $\mu_\beta$, for any weight of the form $\omega(\log(\sigma)+ \log(\omega))$, and the same bound would have been obtained. However, the calculation would have been less straightforward.

Thus, we have a lower bound on the optimal constant whenever the measures satisfy a logarithmic Sobolev inequality. The difficult part of this work, as for any functional inequality, is to show that they really satisfy this inequality and to obtain an upper bound on the optimal constant. For this purpose we use the curvature dimension criterion. 

Let $\nu$ be a probability measure on $\R^n$ and  $L$ diffusion generator, defined on $\CC_c^\infty(\R^n)$, symmetric with respect to $\nu$. We define the carré-du-champ $\Gamma$, and the iterated carré-du-champ operator $\Gamma_2$ on $\CC^\infty_c(\R^n)$ by
\[2\Gamma(f) = L(f^2)-2fLf,\quad 2\Gamma_2(f) = L\Gamma(f) - 2\Gamma(f, Lf).\]
We say that $(L,\Gamma, \nu)$ satisfies the curvature-dimension criterion, or Bakry-\'Emery criterion, $CD(\rho,\infty)$, with $\rho\in\R$, if $\Gamma_2\geq \rho\Gamma$. This criterion was developed by Bakry and \'Emery in \cite{BE}. It is sufficient to prove  functional inequalities such as Poincaré, Beckner and logarithmic Sobolev. For more references about curvature-dimension criterion and functional inequalities, see \cite{BGL}, or \cite{ABCFGMRS} more specifically about the logarithmic Sobolev inequality. The following entropy's representation formula, from \cite{ChL} for instance, sums up the situation.

\begin{prop}[\cite{ChL}]\label{prop:LS}
For all $\rho>0$ and $f\in\CC^\infty_c(\R^n)$, we have
\[\Ent_\nu(f) = \frac{1}{2\rho}\int_{\R^n}\frac{\Gamma(f)}{f}\, d\nu -\frac{1}{\rho}\int_0^{+\infty}\int_{\R^n}\left(\Gamma_2-\rho\Gamma\right)(\log(\BP_tf)\BP_tf\, d\nu\, dt,\]
where $\BP$ denotes the semigroup associated to $L$. In particular, if $(L,\Gamma, \nu)$ satisfies $CD(\rho,\infty)$, with $\rho>0$, then $\nu$ satisfies the weighted logarithmic Sobolev inequality
\[\Ent_\nu(f^2) \leq \frac{2}{\rho}\int_{\R^n} \Gamma(f)\, d\nu.\]
\end{prop}

We aim to prove a curvature-dimension criterion for an appropriate generator $L$, associated to rescaled generalised Cauchy measures. We use a generator whose carré-du-champ $\Gamma$ brings the correct weight $\Omega_\sigma =\omega_\sigma\log(\omega_\sigma)$. Let us define $L$ on $\CC_c^\infty(\R^n)$ by
\[Lf = \omega_\sigma\log(\omega_\sigma) \Delta f +2(1-(\beta-1)\log(\omega_\sigma))\langle x, df\rangle.\]
It is symmetric with respect to $\mu_{\beta,\sigma}$ and can be extended to a unique self-adjoint operator on a domain of $\L^2(\mu_{\beta,\sigma})$. Let us introduce some notations. For $v,w\in\R^n$, we denote by $v\otimes w$ the tensor $(v\otimes w)_{ij} = v_iw_j$ and by $v\odot w = (v\otimes w + w\otimes v)/2$, its symmetric part. For a matrix $M\in\MM_n(\R)$, we denote by $\|M\|_{HS}$ the Hilbert-Schmidt norm, defined as
\[\|M\|_{HS}^2 = \sum_{1\leq i,j\leq n} |M_{ij}|^2.\]

The operator $L$ is an appropriate choice to obtain weighted logarithmic Sobolev inequalities, as shown in the following proposition.

\begin{prop}
For all $f\in\CC_c^\infty(\R^n)$, we have 
\begin{align*}
\Gamma(f) = & \Omega_\sigma|df|^2,\\
\Gamma_2(f)= & \left\|\Omega_\sigma\Hess(f) + \nabla \Omega_\sigma\odot\nabla f-\frac{\langle d\Omega_\sigma, df\rangle}{2}\id\right\|^2_{HS}\\
& +[(n-2)\log(\omega_\sigma)^2 -2(2\beta-n)\log(\omega_\sigma)+n-2]\left[|df|^2|x|^2-\langle df, x\rangle^2\right]\\
&+(2\beta-n)\Gamma(f)\\
& +[-(n-2)(\omega_\sigma-\sigma) + \sigma(2\beta+n-2)\log(\omega_\sigma)^2-2\sigma(\beta-n+1)\log(\omega_\sigma)]|df|^2.
\end{align*}

\end{prop}

\begin{proof}
The computation for $\Gamma$ is quite direct and we will focus on $\Gamma_2$ only. We have 
\begin{align*}
2\Gamma_2(f)=
&  \Omega_\sigma\Delta(\Omega_\sigma|df|^2) +2[1-(\beta-1)\log(\omega_\sigma)]\langle x, d(\Omega_\sigma|df|^2)\rangle\\ 
& -2\Omega_\sigma\langle df, d(\Omega_\sigma\Delta f+2[1-(\beta-1)\log(\omega_\sigma)]\langle x, df\rangle)\rangle\\
= & \Omega_\sigma\Delta\Omega_\sigma|df|^2 + 2 \Omega_\sigma\langle d|df|^2, d\Omega_\sigma\rangle +\Omega_\sigma^2\Delta|df|^2\\ & +2[1-(\beta-1)\log(\omega_\sigma)]\langle d\Omega_\sigma, x\rangle|df|^2 + 2[1-(\beta-1)\log(\omega_\sigma)]\Omega_\sigma\langle d|df|^2, x\rangle\\
& -2\Omega_\sigma\Delta f\langle df, d\Omega_\sigma\rangle -2\Omega_\sigma^2\langle d\Delta f, df\rangle +4(\beta-1)\log(\omega_\sigma)\langle df, d\omega_\sigma\rangle\langle df, x\rangle\\
& -4\Omega_\sigma[1-(\beta-1)\log(\omega_\sigma)]\langle df, d\langle df, x\rangle \rangle\\
\end{align*}

In order to simplify this expression, we use the Bochner formula (see \cite{BGL} for instance), for the Laplacian in $\R^n$
\[\Delta |df|^2-2\langle d\Delta f, df\rangle = 2\|\Hess(f)\|^2_{HS},\]
and the identity
$\frac{1}{2}\langle d|df|^2, x\rangle =  \Hess(f)(\nabla f, x) = \langle d\langle df, x\rangle, df\rangle - |df|^2$. Hence, we have 
\begin{align*}
2\Gamma_2(f) 
= & 2\|\Omega_\sigma\Hess(f)\|^2_{HS} +4\Omega_\sigma\Hess(f) \left(\nabla f, \nabla\Omega_\sigma\right) -2\Omega_\sigma\Delta f\langle df, d\Omega_\sigma\rangle\\
&+8(\beta-1)\log(\omega_\sigma)\langle df,x\rangle^2\\
&+ \left[\Omega_\sigma\Delta\Omega_\sigma + 2(1-(\beta-1)\log(\omega_\sigma))\langle d\Omega_\sigma, x\rangle\right.\\ 
&\quad\quad\left.- 4\Omega_\sigma(1-(\beta-1)\log(\omega_\sigma))\right]|df|^2\\
\end{align*}
%

Because of the second and third terms, this expression is not a combination of non-negative terms, up to multiplicative coefficients. To deal with these terms, we use the following factorisation formula. For $M\in\MM_n(\R)$ symmetric, and $v,w\in\R^n$, we have
\begin{align*}
&\left\|M +v\odot w- \frac{\langle v,w\rangle}{2}\id\right\|^2_{HS}
= \sum_{1\leq i,j\leq n}\left(M_{ij} + \frac{v_i w_j + v_j w_i}{2} - \frac{\langle v,w\rangle}{2}\delta_{ij}\right)^2\\
= &\sum_{1\leq i,j\leq n} M_{ij}^2 + M_{ij}(v_i w_j + v_j w_i)-M_{ij}\langle v,w\rangle\delta_{ij} +\frac{(v_i w_j + v_j w_i)^2}{4}\\
 &\quad - \frac{v_i w_j + v_j w_i}{2}\langle v,w\rangle\delta_{ij} + \frac{\langle v,w\rangle^2}{4}\delta_{ij}\\
 = & \left\|M \right\|^2_{HS} +2\langle Mv,w\rangle - \Tr(M)\langle v,w\rangle + \frac{|v|^2|w|^2}{2} + \frac{n-2}{4}\langle v,w\rangle^2\\
\end{align*}
We apply this formula to $M= \Omega_\sigma\Hess(f)$, $v=\nabla f$ and $w=\nabla\Omega_\sigma$. 
Hence, we have
\begin{align*}
\Gamma_2(f) 
= & \left\|\Omega_\sigma\Hess(f)+\nabla f\odot\nabla \Omega_\sigma -\frac12\langle df, d\Omega_\sigma\rangle\id \right\|^2_{HS}\\
&-\frac{n-2}{4}\langle df,d\Omega_\sigma\rangle^2 +4(\beta-1)\log(\omega_\sigma)\langle df,x\rangle^2\\
&+ \left[\frac{1}{2}\Omega_\sigma\Delta\Omega_\sigma -\frac{1}{2}|d\Omega_\sigma|^2 + (1-(\beta-1)\log(\omega_\sigma))\langle d\Omega_\sigma, x\rangle\right.\\ 
&\quad\quad\left.- 2\Omega_\sigma(1-(\beta-1)\log(\omega_\sigma))\vphantom{\frac12}\right]|df|^2\\
\end{align*}
Then, we use the identity $|x|^2 = \omega_\sigma-\sigma$ and we replace  $d\Omega_\sigma$ and $\Delta\Omega_\sigma$ by their expressions in terms of $\omega_\sigma$ 
\[d\Omega_\sigma = 2(1+\log(\omega_\sigma))x,\quad \Delta\Omega_\sigma = 2n(1+\log(\omega_\sigma)) + 4\frac{\omega_\sigma-\sigma}{\omega_\sigma}.\]
We obtain 
\begin{align*}
\Gamma_2(f) 
= & \left\|\Omega_\sigma\Hess(f)+\nabla f\odot\nabla \Omega_\sigma -\frac12\langle df, d\Omega_\sigma\rangle\id \right\|^2_{HS}\\
&+\left[4(\beta-1)\log(\omega_\sigma)-(n-2)(1+\log(\omega_\sigma))^2\right]\langle df,x\rangle^2 \\
&+ \left[(n-2)\omega_\sigma\log^2(\omega_\sigma) +2\beta\sigma\log^2(\omega_\sigma) + (n-2\beta)\omega_\sigma\log(\omega_\sigma) \right.\\ 
&\quad\quad\left. +2(\beta-1)\sigma\log(\omega_\sigma)\right]|df|^2\\
\end{align*}

For the last step, we note that the second term, $\langle df,x\rangle^2$, is weighted by a  coefficient
\[[-(n-2)\log(\omega_\sigma)^2 +2(2\beta-n)\log(\omega_\sigma)-(n-2)],\] 
which is negative for sufficiently large $|x|$ (except in dimensions $1$ and $2$, which present other difficulties). In order to circumvent this issue, we introduce the compensated expression $|df|^2|x|^2 - \langle df,x\rangle^2$. It is non-negative and its coefficient can be also non-negative under suitable assumptions on $\beta$ and $\sigma$. 
This concludes the proof.
\end{proof}

Let us notice that the Hilbert-Schmidt term can be interpreted as the Hessian of $f$ for a conformal change of metric $g= \Omega_\sigma^{-1}\langle\cdot,\cdot\rangle$. This could lead to a new approach, using metric changes.  However, unlike the work on the Poincaré inequality in \cite{Hug}, here the adapted coordinates to be dealt with do not seem to be explicit.
\section{One-dimensional case}
\label{Sec:1d}

The goal of this section is to obtain a weighted logarithmic Sobolev inequality in the one-dimensional case. We explain that the result is optimal from several perspectives. This result is used in section \ref{Sec:2d} to obtain concentration properties.

First of all, in dimension $1$, the expression of the iterated carré-du-champ operator can be simplified. In fact, the second term vanishes. We obtain the following expressions.

\begin{prop}\label{prop:G1d}
For all $f\in\CC_c^\infty(\R)$, we have
\[\Gamma_2(f) = \left(\Omega_\sigma f'' + \frac12 \Omega_\sigma' f'\right)^2 +(2\beta-1)\Gamma(f) + R_\sigma(\omega_\sigma) f'^2,\]
with $R_\sigma(t) =t-\sigma +(2\beta-1)\sigma\log(t)^2-2\beta\sigma\log(t)$. 
\end{prop}

Therefore, in order to prove a $CD(2\beta-1,\infty)$ criterion, it is sufficient to prove that the rest $R_\sigma$ is positive.  In dimension $1$, this is true, at least for some $\sigma$ large enough.

\begin{theorem}\label{prop:CD1d}
Let $\sigma_\beta= \exp(2\beta/(2\beta-1) )$. Then $(L,\Gamma, \mu_{\beta,\sigma})$ satisfies the criterion $CD(2\beta-1, \infty)$ if and only if $\sigma\geq\sigma_\beta$.
\end{theorem}

\begin{proof}
Firstly, we show that the remaining term $R_\sigma(t)$ is non-negative for all $t\geq\sigma_\beta$. On the one hand, we have
\[R_\sigma(\sigma) = \sigma\log(\sigma)((2\beta-1)\log(\sigma)-2\beta).\]
Then, $R_\sigma(\sigma)\geq0$ if and only if $\sigma\geq\sigma_\beta$. 
On the other hand, $R_\sigma$ is differentiable and we have
\[R'_\sigma(t) = 2(2\beta-1)\sigma\frac{\log(t)}{t} -2\beta\frac{\sigma}{t}+1.\]
Then for all $\sigma\geq\sigma_\beta$, for all $t\geq\sigma$, we have $R'_\sigma(t)\geq1$. 
It follows that $R_\sigma$ is non-negative for all $\sigma\geq\sigma_\beta$ and from Proposition \ref{prop:G1d}, we have $\Gamma_2(f)\geq (2\beta-1)\Gamma(f)$, for all $f\in\CC^\infty_c(\R)$. 
\end{proof}

Now, we apply the rescaling argument of Lemma \ref{prop:rescale} so as to prove the logarithmic Sobolev inequality for the generalised Cauchy measure.
 
\begin{coro}\label{prop:LS1d}
For all $\beta>1/2$, for all $f\in\CC^\infty_c(\R)$, we have
\[\Ent_{\mu_\beta}(f^2)\leq \frac{2}{2\beta-1}\int_\R f'^2\omega\left(\frac{2\beta}{2\beta-1} +\log(\omega)\right)\, d\mu_\beta.\]
\end{coro}

Let us notice that this inequality is also optimal insofar as letting $\beta$ go to $+\infty$, we recover the optimal Gross' logarithmic Sobolev inequality for Gaussian measures. Indeed, the rescaling argument, from Lemma \ref{prop:rescale}, holds for the measure $\mu_{\beta,2\beta}$ and we have
\[\Ent_{\mu_{\beta,2\beta}}(f^2)\leq \frac{2}{2\beta-1}\int_\R \left(f'(x)\right)^2(2\beta+x^2)\left(\frac{2\beta}{2\beta-1} +\log\left(1+\frac{x^2}{2\beta}\right)\right)\, \mu_{\beta,2\beta}(dx).\]
Then, as $\beta$ goes to $+\infty$, we recover the optimal inequality
\[\Ent_\gamma(f^2)\leq 2\int f'^2\, d\gamma.\]

This result is also optimal in the sense that the constant $2\beta-1$ and Theorem \ref{prop:CD1d} are optimal. However, the question of whether it is possible to reach the weight $\omega(1+\log(\omega))$ is still open. This question is particularly important as $\beta$ goes to $1/2$ for which both, the constant $2/(2\beta-1)$ and $\log{\sigma_\beta}$, explode. Perhaps it is possible to reduce the combined explosion rate by using a sharper criterion.

\section{General case}\label{Sec:nd}
The goal of this section is to obtain a weighted logarithmic Sobolev inequality in dimension $n\geq3$. Actually, in the general case $n\geq 2$, the curvature-dimension criterion is not optimal in the sense that it does not allow to recover the optimal lower bound, but it does not provide relevant inequalities for all $n/2<\beta$. Indeed, we can still recover an inequality with an explicit constant, for small $\beta$, but for larger $\beta$, the lower bound $\sigma_\beta$ becomes unreasonably high.

To begin with, let us recall the expression of the iterated carré-du-champ operator $\Gamma_2$. In Section \ref{Sec:Cauchy}, we proved
\[\Gamma_2(f) = \|M(f)\|^2_{HS} + A_\sigma(\log(\omega_\sigma))\left[|df|^2|x|^2-\langle df, x\rangle^2\right] + (2\beta-n)\Gamma(f) + R_\sigma(\omega_\sigma)|df|^2,\]
with 
\begin{align*}
M(f) &= \Omega_\sigma\Hess(f) + \nabla\Omega_\sigma\odot\nabla f-\frac12\langle d\Omega_\sigma, df\rangle\id\\
A_\sigma(t)&=(n-2)t^2 -2(2\beta-n)t+n-2\\
R_\sigma(t)&= -(n-2)(t-\sigma) + \sigma(2\beta+n-2)\log(t)^2-2\sigma(\beta-n+1)\log(t).
\end{align*}

There are two main differences between the one-dimensional case and the general case. First, there is an additional coefficient, $A_\sigma$, whose positivity must be checked. This imposes new constraints on $\sigma$. Second, the remaining term, $R_\sigma$, is no longer non-negative, even for large $\sigma$. Indeed, its highest order term, $(2-n)t$, is non-positive. To get around this obstacle, we have to sacrifice the optimal constant $2\beta-n$ and look for a $CD(2\beta-n-\varepsilon,\infty)$ criterion, for some $\varepsilon>0$.  Nevertheless, before using of these considerations, we need to exclude the case $n=2$.

\begin{prop}
For $n=2$, $(L,\Gamma, \mu_{\beta,\sigma})$ does not satisfy any $CD(\rho,\infty)$ criterion for any $\rho>0$ and $\sigma>1$.
\end{prop}

\begin{proof}
After simplifications, the iterated carré-du-champ has the expression
\begin{align*}
\Gamma_2(f)
=& \|M(f)\|^2_{HS} + 4(\beta-1)\log(\omega_\sigma)\langle df, x\rangle^2\\
& +2\left[(1-\beta)+\sigma\frac{\beta\log(\omega_\sigma)+\beta-1}{\omega_\sigma}\right]\Gamma(f)\\
\end{align*}
For all $x\in\R^2$, it is always licit to find a function $f$ such that $M(x)=0$,  $\langle df, x\rangle=0$ and $\Gamma(f)(x)\neq0$. As in dimension $2$, $\beta>1$, for $|x|$ sufficiently large, it is always possible to find $f$ such that $\Gamma_2(f)(x)<0$.
\end{proof}

Let us remark that a similar phenomenon can also be observed for the generator $\omega\Delta -2(\beta-1)x$, associated with the Poincaré and Beckner inequalities. This generator does not satisfy any $CD(\rho,\infty)$ criterion for positive $\rho$ in any dimension, but it does satisfy a $CD(\rho,N)$ criterion for some $N<0$, except in dimension $2$. We deal with the pitfall of dimension $2$ in Section \ref{Sec:2d}, using a special argument. For the rest of this section, we  assume that $n\geq3$. We obtain a first constraint on $\sigma$ with the control of $A_\sigma$.

\begin{prop}
Let $n\geq3$. If $n/2<\beta\leq n-1$ then for all $\sigma>1$, $A_\sigma(\log(\omega_\sigma))\geq0$. On the other hand, if $\beta > n-1$, then $A_\sigma(\log(\omega_\sigma))$ is non-negative if and only if 
\[\log(\sigma)\geq\frac{2\beta-n +\sqrt{(\beta-1)(\beta-n+1)}}{n-2}.\]
\end{prop}

\begin{proof}
The coefficient $A$ is polynomial. Its discriminant is $\Delta=16(\beta-1)(\beta-n+1)$.
Then, for $n/2<\beta\leq n-1$ the discriminant is negative and $A_\sigma$ is positive without any restriction on $\sigma$. For $\beta > n-1$, $A_\sigma(t)$ has two roots. The greatest root is
\[t_+ =\frac{2\beta-n +\sqrt{(\beta-1)(\beta-n+1)}}{n-2}.\] 
So $A_\sigma(\omega_\sigma)$ is non-negative if and only if $\log(\sigma)$ is larger than this root.
\end{proof}

Let us remark that $A_\sigma(t)$ is also positive for $t$ smaller than the second root, but this condition is not interesting since we are looking for the non-negativity of $A_\sigma(\log(\omega_\sigma(x)))$ for all $x\in\R^n$. This proposition suggests that we can get a reasonable result for small $\beta$. On the other hand, for large $\beta$ the lower bound on $\sigma$ explodes as $\beta$ goes to infinity. This prevents us from using the rescaling argument to recover Gross' inequality for Gaussian measures.

\begin{theorem}\label{prop:dim_n}
Let $\log(\sigma_\beta)=2(n-2)/(2\beta-n)$. Then for all $n/2<\beta\leq n-1$, for all $f\in\CC^\infty_c(\R^n)$, we have
\[\Ent_{\mu_\beta}(f^2)\leq \frac{4}{2\beta-n}\int_{\R^n}|\nabla f|^2\omega(\log(\sigma_\beta) +\log(\omega))\, d\mu_\beta.\]
\end{theorem}

\begin{proof}
Under the assumption $n/2<\beta\leq n-1$, we have proved that
\[\Gamma_2(f)\geq (2\beta-n)\Gamma(f) + R_\sigma(\omega_\sigma)|df|^2.\]
However,  $R_\sigma$ is not non-negative. Let $\varepsilon>0$ a free parameter. We are looking for conditions on $\sigma$ and $\varepsilon$ under which the new remaining term $\RR_{\sigma,\varepsilon}(t)=\varepsilon+R_\sigma(t)/(t\log(t))$ is non-negative. For every couple $(\sigma, \varepsilon)$ satisfying this condition, we have proved that $\mu_{\beta,\sigma}$ satisfies a $CD(2\beta-n-\varepsilon,\infty)$ criterion and so, $\mu_\beta$ satisfies the inequality
\[\Ent_{\mu_\beta}(f^2)\leq \frac{2}{2\beta-n-\varepsilon}\int_{\R^n}|df|^2\omega(\log(\sigma)+\log(\omega))\, d\mu_\beta.\]
That is why, among all the possible couple $(\sigma,\varepsilon)$, we select a solution such that the quantity $\log(\sigma)/(2\beta-n-\varepsilon)$ is as small as possible, but also explicitly computable. We have
\[\RR_{\sigma,\varepsilon}(t) = \varepsilon -(n-2)\frac{t-\sigma}{t\log(t)} + \sigma(2\beta+n-2)\frac{\log(t)}{t}-2\sigma(\beta-n+1)\frac{1}{t}.\]

As $n\geq3$ and $\beta\leq n-1$, the two last terms are non-negative. Besides, $t\geq\sigma$, so it is sufficient to have \[\varepsilon +\frac{2-n}{\log(\sigma)}\frac{(t-\sigma)}{t}\geq0.\]
Hence, it is sufficient to impose the relation $\varepsilon = (n-2)/\log(\sigma)$. Then, optimising the quantity $\log(\sigma)/(2\beta-n-\varepsilon)$, we obtain
\[\log(\sigma_\beta)=\frac{2(n-2)}{2\beta-n}\quad\text{and}\quad\varepsilon=\frac12(2\beta-n).\]
This ends the proof.
\end{proof}

Unlike Corollary \ref{prop:LS1d} in dimension $1$, this result is probably not optimal, in both perspective. Firstly, it misses the lower bound from Proposition \ref{prop:lower}, by a factor of $2$. Note that the constant $\log(\sigma_\beta)$ is also twice as large as expected from Corollary \ref{prop:LS1d}, as $\beta$ goes to $n/2$. These two points could probably be improved as our proof only proceeds from sufficient conditions at each step. However, it seems very unlikely that an optimal solution $(\sigma,\varepsilon)$ that minimises our criterion under the condition $\RR_{\sigma,\varepsilon}\geq0$, whether it exists or not, has a closed-form expression. In any case, we want to stress again that a mere curvature-dimension argument is not sharp enough to recover the lower bound $2/(2\beta-n)$ in dimension $n\geq2$.  Moreover, this result is only valid for $\beta\leq n-1$ and thus does not allow to recover an inequality for the Gaussian case.

\section{Two-dimensional case: a tensorisation argument}
\label{Sec:2d}

As explained in the previous section, the curvature-dimension approach does not provide any result in two dimensions. In this section, we develop a different argument based on a tensorisation argument: a two-dimensional measure can be seen as a product of two one-dimensional measures. This approach is based on the so-called Marton's coupling from \cite{Mart}. This argument has been develop to obtain cost-information inequalities in \cite{DGW} for dependent sequences and to obtain logarithmic Sobolev inequalities in \cite{BB} for Markov chains. 

For all $\beta>1$, we denote by $\mu^{(2)}_\beta$ the generalised Cauchy measure in $\R^2$ and $\sigma :x\in\R\mapsto 1+x^2$. The measure $\mu^{(2)}_\beta$ can be factorised as
\[\mu^{(2)}_\beta(dxdy) = \mu_{\beta,\sigma_x}(dy)\mu_{\beta-1/2}(dx).\]

We recall that according to Corollary \ref{prop:LS1d}, $\mu_{\beta}$ satisfies a logarithmic Sobolev inequality with constant $C_{\beta} = 2/(2\beta-1)$ and weight  $x\mapsto\sigma_x\left(\kappa + \log(\sigma_x)\right)$ (with $\kappa = 2\beta/(2\beta-1)$). Thanks to a rescaling argument, similar to Lemma \ref{prop:rescale}, the measure $\mu_{\beta,\sigma_x}(dy)$ also satisfies a logarithmic Sobolev inequality with constant $C_\beta$ but weight $(\sigma_x + y^2)\left(\kappa + \log(1+y^2/\sigma_x)\right)$. We are going to make good use of our optimal logarithmic Sobolev inequality in dimension one, in both a direct and indirect way. Firstly, we prove a Gaussian concentration property under a weighted logarithmic Sobolev inequality. This kind of result is very classical without weight (see \cite{BGL} or \cite{Led3} for instance). Its generalisation to weighted inequalities might appear in former works although we could not find it.  

\begin{lemme}
Let $\Omega:\R^n\to\R$ a positive weight and $\nu$ a probability measure on $\R^n$. Suppose that $\nu$ satisfies a weighted logarithmic Sobolev inequality with weight $\Omega$ and constant $C>0$, then, for all $f\in\CC^1(\R^n)$ such that $\|\Omega |df|^2\|_\infty<\infty$, for all $t\in\R$, we have
\[\nu\left(e^{tf}\right) \leq \exp\left(t\nu(f) + \frac{C\|\Omega |df|^2\|_\infty}{4} t^2\right).\]
\end{lemme}

\begin{proof}
The proof is directly adapted from the classical Herbst argument. Let $f$ be compactly supported. The general case can be reached by a density argument. We set $\psi_t = \nu\left(e^{2tf}\right)$. We apply the logarithmic Sobolev inequality to $e^{tf}$ to obtain a differential inequality on $\psi$. Indeed, we have
\[\Ent_\nu(e^{2tf}) = t\psi_t'-\psi_t\log(\psi_t)\quad\text{and}\quad
\int_{\R^n}\left|d e^{tf}\right|^2\Omega\, d\nu \leq\|\Omega |df|^2\|_\infty t^2\psi_t.\]
Hence, $\psi$ satisfies the inequality $t\psi'_t - \psi_t\log(\psi_t)\leq C\|\Omega |df|^2\|_\infty t^2\psi_t$. So as to integrate this inequality, we set $\phi_t=\log(\psi_t)/t$, with the initial condition $\phi_0=2\nu(f)$. It satisfies $\phi'_t\leq C\|\Omega |df|^2\|_\infty $,  and we obtain the result.
\end{proof}

Let us notice that, whenever the weight $\Omega$ is not bounded, this concentration result is weaker than the usual Gaussian concentration inequality as it is applied to a more restricted class of functions. However, it is sufficient enough to prove the following $\L^1$-$\L\log(\L)$ inequality.

\begin{lemme}\label{prop:L1-LlogL}
Assume that $\nu$ satisfies the previous hypothesis. For all $f\in\L^2(\nu)$ and $h\in\CC^1$, centred, such that $\|\Omega |dh|^2\|_\infty<\infty$, we have
\[\left|\nu\left(f^2h\right)\right|^2\leq C\|\Omega |dh|^2\|_\infty \nu(f^2)\Ent_\nu(f^2).\]
\end{lemme}

\begin{proof}
Let us denote $\Lambda_t = \nu\left(e^{th}\right)$. By definition, we have $\nu(\exp(th-\log(\Lambda_t)))=1$ for all $t$. According to the Donsker-Varadhan variational representation of the entropy, we have
\[\nu((th-\log(\Lambda_t))f^2)\leq \Ent_\nu(f^2).\]
From the previous lemma, we know that $\log(\Lambda_t)\leq C\|\Omega |dh|^2\|_\infty t^2/4$.
Thus, for all $t$, we have
\[\frac{C\|\Omega |dh|^2\|_\infty\nu(f^2)}{4}t^2-\nu(f^2h)t+\Ent_\nu(f^2)\geq0.\]
In particular, it holds for $t=\left(2\nu(f^2h)\right)/\left(C\|\Omega |dh|^2\|_\infty \nu(f^2)\right)$. Hence, the result is proved.
\end{proof}

Now we have all the necessary argument to prove a logarithmic Sobolev inequality for generalised Cauchy measures in dimension $n=2$.

\begin{theorem}
Let $\kappa=2\beta/(2\beta-1)$. For all $1<\beta$ and $f\in\CC^\infty_c(\R^2)$, we have
\[\Ent_{\mu^{(2)}}\left(f^2\right)\leq \frac{4\kappa^3 +2 +\frac{\kappa}{2\beta} + 4\kappa\sqrt{\left( \kappa^2+\frac{1}{8\beta}\right)^2+\kappa}}{2\beta-1} \int_{\R^2}|df|^2\omega\left(\kappa+\log(\omega)\right)^2\, d\mu^{(2)}.\]
\end{theorem}

\begin{proof}
Let $f\in\CC^\infty_c(\R^2)$ such that $\mu_\beta^{(2)}(f^2)=1$. First, we establish a tensorisation formula for the entropy. We have
\begin{align*}
\Ent_{\mu_\beta^{(2)}}(f^2) 
&= \int_\R\int_\R f^2\log(f^2)\, \mu_{\beta,\sigma_x}(dy)\, \mu_{\beta-1/2}(dx)\\
&= \int_\R \Ent_{\mu_{\beta,\sigma_x}}(f^2)\, \mu_{\beta-1/2}(dx) + \int_\R F^2\log(F^2)\, \mu_{\beta-1/2}(dx)
\end{align*}
where the function $F\in\CC^\infty_c(\R)$ is defined by $F^2(x) = \int_\R f^2(x,y)\, \mu_{\beta,\sigma_x}(dy)$. Since this function satisfies $\mu_{\beta-1/2}(F^2)=1$, we obtain the following tensorisation formula 
\begin{equation}\label{eq:ent2}
\Ent_{\mu_\beta^{(2)}}(f^2) =\int_\R \Ent_{\mu_{\beta,\sigma_x}}(f^2)\, \mu_{\beta-1/2}(dx) +\Ent_{\mu_{\beta-1/2}}(F^2).
\end{equation}
Now, we work on the entropy of a one-dimensional generalised Cauchy measure. We need to understand $F'$ in order to apply the logarithmic Sobolev inequality from Section \ref{Sec:1d}. From its definition, we have
\begin{align*}
2F(x)F'(x)
&= (F^2)'(x)\\
&=\partial_x\int_\R f^2(x,y) (\sigma_x+y^2)^{-\beta}\sigma_x^{-1/2+\beta}\frac{dy}{Z_\beta}\\
&=2\int_\R f\partial_x f\, \mu_{\beta,\sigma_x}(dy) + \int_\R f^2 h_x\, \mu_{\beta,\sigma_x}(dy)\\
\end{align*} 
where 
\[h_x(y) 
= (\sigma_x+y^2)^{\beta}\sigma_x^{1/2-\beta} \partial_x\left((\sigma_x+y^2)^{-\beta}\sigma_x^{-1/2+\beta}\right)
= \frac{[(2\beta-1)y^2-\sigma_x]x}{\sigma_x(\sigma_x+y^2)}
\]
which morally represent the differential of the disintegration kernel. It follows that $\mu_{\beta,\sigma_x}(h)=0$. Moreover, $\mu_{\beta,\sigma_x}(dy)$ satisfies a logarithmic Sobolev inequality with constant $C_\beta$ and weight $\Omega_x(y) =(\sigma_x+y^2)(\kappa +\log(1+y^2/\sigma_x))$. For all $x\in\R$, the function $h_x$ satisfies the boundedness assumption of Lemma \ref{prop:L1-LlogL}. Indeed, we have
\begin{align*}
(\partial_y h_x(y))^2\Omega_x(y)
&=16\beta^2 \frac{x^2y^2(\sigma_x+y^2)(\kappa+\log(1+y^2/\sigma_x))}{(\sigma_x+y^2)^4}\\
&\leq 16\beta^2\frac{\kappa+\log(1+y^2/\sigma_x)}{\sigma_x+y^2}\\
&\leq \frac{16\beta^2\kappa}{\sigma_x}
\end{align*}
So, for all $x\in\R$, $\|\partial_yh_x\Omega_x\|_\infty\leq \frac{16\beta^2\kappa}{\sigma_x}$ and the lemma applies. Together with the Cauchy-Schwarz inequality, we have
\[|F(x)F'(x)|
\leq |F(x)|\left(\int |\partial_x f|^2 \mu_{\beta,\sigma_x}(dy)\right)^{1/2}
+ |F(x)|\left(\frac{4\beta^2\kappa C_\beta}{\sigma_x} \Ent_{\mu_{\beta,\sigma_x}}(f^2)\right)^{1/2}.\]
A last, from Young's inequality, we have for all $\lambda>0$
\[|F'(x)|^2\leq (1+\lambda)\int |\partial_x f|^2\, \mu_{\beta,\sigma_x}(dy) + \left(1+\frac{1}{\lambda}\right)\frac{4\beta^2\kappa C_\beta}{\sigma_x}\Ent_{\mu_{\beta,\sigma_x}}(f^2).\]
Now that we dispose a precise control over $|F'|^2$, we apply the logarithmic Sobolev inequality for $\mu_{\beta-1/2}$
\begin{align*}
\Ent_{\mu_{\beta-1/2}}(F^2) 
\leq & C_{\beta-1/2}\int |F'|^2\sigma_x\left(\kappa+\log(\sigma_x)\right)\, \mu_{\beta-1/2}(dx)\\
\leq & C_{\beta-1/2}(\lambda+1)\int|\partial_xf|^2\sigma_x\left(\kappa+\log(\sigma_x)\right)\, d\mu_{\beta}^{(2)}\\
& + 4\beta^2\kappa C_\beta C_{\beta-1/2}\frac{\lambda+1}{\lambda}\int \Ent_{\mu_{\beta,\sigma_x}}(f^2)\left(\kappa+\log(\sigma_x) \right)\,\mu_{\beta-1/2}(dx)
\end{align*}
From \eqref{eq:ent2}, we have
\begin{align*}
\Ent_{\mu_\beta^{(2)}}(f^2) 
\leq& \left(1+4\beta^2\kappa C_\beta C_{\beta-1/2}\frac{\lambda+1}{\lambda}\right)\int \Ent_{\mu_{\beta,\sigma_x}}(f^2)(\kappa+\log(\sigma_x))\, \mu_{\beta-1/2}(dx)\\
&+ C_{\beta-1/2}(\lambda+1)\int|\partial_xf|^2\sigma_x(\kappa+\log(\sigma_x))^2\, d\mu_{\beta}^{(2)}\\
\end{align*}
Then, we apply the logarithmic Sobolev inequality for $\mu_{\beta,\sigma_x}$
\begin{align*}
\int& \Ent_{\mu_{\beta,\sigma_x}}(f^2)\left(\kappa+\log(\sigma_x) \right)\, \mu_{\beta-1/2}(dx)\\
& \leq C_\beta\int_\R |\partial_yf|^2(\kappa+\log(\sigma_x) )(\sigma_x+y^2)(\kappa +\log(1+\frac{y^2}{\sigma_x}))\, d\mu^{(2)}_\beta.
\end{align*}
Together with the previous inequality, we obtain
\begin{align*}
\Ent_{\mu_\beta^{(2)}}(f^2) 
\leq& \MM_\beta(\lambda)\int|df|^2\omega\left(\kappa+\log(\omega)\right)^2\, d\mu_\beta^{(2)}
\end{align*}
with 
\[\MM_\beta(\lambda)=\max\left\{C_{\beta-1/2}(\lambda+1),C_\beta \left(1+4\beta^2\kappa C_\beta C_{\beta-1/2}\frac{\lambda+1}{\lambda}\right) \right\}.\]
To conclude, we optimise over $\lambda>0$. The minimum, $\MM_\beta$, always exists and is reached for 
\[\lambda_0= \frac{4\beta^2\kappa C_\beta^2C_{\beta-1/2} + C_\beta - C_{\beta-1/2}-\sqrt{\Delta}}{2C_{\beta-1/2}},\]
with $\Delta = \left(4\beta^2\kappa C_\beta^2C_{\beta-1/2} + C_\beta - C_{\beta-1/2}\right)^2+16\beta^2\kappa C_\beta^2C_{\beta-1/2}^2$. Let us remark that $\beta C_\beta =\kappa$. It follows that 
\[\Delta =C_{\beta-1/2}^2 \left[\left(4\kappa^3 + \frac{C_\beta}{C_{\beta-1/2}} -1\right)^2+16\kappa^3\right] = 16\kappa^2 C^2_{\beta-1/2} \left[\left( \kappa^2+\frac{1}{8\beta}\right)^2+\kappa\right] \]
To conclude, we have
\[\MM_\beta 
= C_\beta (\lambda_0+1)
= C_\beta\frac{4\kappa^3C_{\beta-1/2} + C_\beta + C_{\beta-1/2}-\sqrt{\Delta}}{2C_{\beta-1/2}}.\]
This ends the proof.
\end{proof}

One can notice that this method is not optimal. For instance, the simplification of the weight, at the last step, is not sharp. A more precise weight would have been
\[\omega(x,y)\left(\kappa + \log(1+x^2)\right)\left(\kappa + \log (1+y^2/(1+x^2))\right)\]
which has order $\omega\log(\omega)$ on each direction $y=cx$ but is globally of order $\omega\log(\omega)^2$ (along $y=x^2$ for instance). Actually, by iterating the process in larger dimensions, we get a logarithmic Sobolev inequality for the generalised Cauchy measure in $\R^n$, with explicit constant but weight $(1+|x|^2)\left(\kappa+\log(1+|x|^2)\right)^n$. Nonetheless, this method has several main features. Firstly, it provides a logarithmic Sobolev inequality for the whole range of parameters $\beta$, unlike the curvature-dimension argument from Section \ref{Sec:nd}. Besides, the constant is explicit, unlike by the Lyapunov's method for instance. Furthermore, the constant has the right order of magnitude as $\beta$ goes to $+\infty$. Indeed, the rescaled measure $\mu^{(2)}_{\beta, 2\beta}$ satisfies
\[\Ent_{\mu^{(2)}_{\beta, 2\beta}}(f^2)\leq 2\beta\MM_\beta\int_{\R^2}|df|^2 (1+(x^2+y^2)/2\beta)(\kappa +\log(1+(x^2+y^2)/2\beta))^2\, d\mu^{(2)}_{\beta, 2\beta}.\]
Therefore, we can recover  Gross' inequality for Gaussian measure, with a non-optimal constant,
\[\Ent_{\gamma^{(2)}}(f^2)\leq (6+4\sqrt{2})\int |df|^2\,d\gamma^{(2)}.\]

\section*{Acknowledgement}
We would like to thank François Bolley for his encouragement and for our helpful discussions. This research is supported by the Centre Henri Lebesgue (ANR-11-LABX-0020-0).

\bibliographystyle{plain}
\bibliography{biblio}

\end{document}